\documentclass[a4paper,reqno]{amsart}

\usepackage{graphicx}
\usepackage{amsmath}
\usepackage{amssymb}
\input xy
\xyoption{all}

\usepackage{color}

\numberwithin{equation}{section}

\theoremstyle{plain}

\newtheorem{thm}{Theorem}[section]
\newtheorem{cor}[thm]{Corollary}
\newtheorem{lem}[thm]{Lemma}
\newtheorem{prop}[thm]{Proposition}

\newtheorem{defn}[thm]{Definition}
\newtheorem{exm}[thm]{Example}

\theoremstyle{remark}
\newtheorem{rem}[thm]{Remark}

\newdir{ >}{{}*!/-5pt/\dir{>}}

\renewcommand{\mod}{\operatorname{mod}\nolimits}

\newcommand{\add}{\operatorname{add}\nolimits}

\newcommand{\id}{\operatorname{id}\nolimits}
\newcommand{\Hom}{\operatorname{Hom}\nolimits}

\newcommand{\End}{\operatorname{End}\nolimits}

\newcommand{\Ker}{\operatorname{Ker}\nolimits}
\newcommand{\Ext}{\operatorname{Ext}\nolimits}

\newcommand{\op}{\operatorname{op}\nolimits}

\newcommand{\Cone}{\operatorname{Cone}\nolimits}
\newcommand{\CoCone}{\operatorname{CoCone}\nolimits}

\newcommand{\M}{\mathcal M}

\newcommand{\B}{\mathcal B}

\newcommand{\oB}{\overline{\B}}
\newcommand{\U}{\mathcal U}
\newcommand{\V}{\mathcal V}
\newcommand{\W}{\mathcal W}
\newcommand{\h}{\mathcal H}
\newcommand{\s}{\mathcal S}
\newcommand{\T}{\mathcal T}
\newcommand{\D}{\mathcal D}

\newcommand{\R}{\mathcal R}

\newcommand{\C}{\mathcal C}

\newcommand{\EE}{\mathbb E}
\newcommand{\STUV}{((\s,\T),(\U,\V))}

\newcommand{\svecv}[2]{\left(\begin{smallmatrix}
      #1 \\
      #2
    \end{smallmatrix}\right)}

\newcommand{\svech}[2]{\left(\begin{smallmatrix}
      #1 & #2
\end{smallmatrix}\right)}

\renewcommand{\emph}{\textit}
\renewcommand{\phi}{\varphi}

\begin{document}

\title{Localizations of the hearts of cotorsion pairs associated with mutations}
\author[Yu Liu]{Yu Liu$^\ast$}
\thanks{$^\ast$Department of Mathematics, Southwest Jiaotong University, Chengdu 611756, China, Email: liuyu86@swjtu.edu.cn, Telephone: +86-15910238534}
\thanks{2010 Mathematics Subject Classification： 16G99, 18E99}


\begin{abstract}
In this article, we study localizations of hearts of cotorsion pairs $(\U,\V)$ where $\U$ is rigid on an extriangulated category $\B$. The hearts of such cotorsion pairs are equivalent to the functor categories over the stable category of $\U$ ($\mod \underline \U$). Inspired by Marsh and Palu \cite{MP}, we consider the mutation (in the sense of \cite{IY}) of $\U$ that induces a cotorsion pair $(\U',\V')$. Generally speaking, the hearts of $(\U,\V)$ and $(\U',\V')$ are not equivalent to each other, but we will give a generalized pseudo-Morita equivalence between certain localizations of their hearts. 
\end{abstract}


\maketitle


\section{Introduction}

When we say localization in this article, we mean Gabriel-Zisman localization which is introduced in \cite{GZ}. A well-known example of such localization is the bounded derived category of a module category, it is a localization of homotopy category of complexes. An example on triangulated categories is given in \cite{BM}. They proved that the category of finite-dimensional modules over the endomorphism algebra of a rigid object in a Hom-finite triangulated category is equivalent to the localization of the category with respect to a certain class of morphisms. More localizations are discussed in \cite{BM2} and also in \cite{MP}.

Exact categories are used widely in representative theory and according to the well-known result by Happel, the stable category of a Frobenius category (a special case of exact category) is triangulated \cite{Ha}. One may ask if we can investigate the similar kinds of  localizations on exact categories. Since extriangulated category \cite{NP} generalizes both triangulated and exact category, it is reasonable to study the subject on this more general structure.

In this article, let $k$ be a field, $(\B,\mathbb{E},\mathfrak{s})$ be an extriangulated category defined in \cite{NP} (see $\S2$ of \cite{NP} for details). Any subcategory discussed in this article will be full, additive and closed under direct sums and isomorphisms. Let $\mathcal P$ (resp.  $\mathcal I$) be the subcategory of projectives (resp. injectives). For a subcategory $\C$, let $\C^{\bot_1}=\{B\in\B \text{ }| \text{ } \EE(\C,B)=0 \}$ and ${^{\bot_1}}\C=\{B\in\B \text{ }| \text{ } \EE(B,\C)=0 \}$.

We first recall the definition of cotorsion pair and its heart.

\begin{defn}\label{extriangulated}
Let $\U$ and $\V$ be subcategories of $\B$ which are closed under direct summands. We call $(\U,\V)$ a \emph{cotorsion pair} if it satisfies the following conditions.
\begin{itemize}
\item[(a)] $\EE(\U,\V)=0$.
\item[(b)] For any object $B\in \B$, there exist two conflations
\begin{align*}
V_B\rightarrowtail U_B\twoheadrightarrow B,\quad
B\rightarrowtail V^B\twoheadrightarrow U^B
\end{align*}
satisfying $U_B,U^B\in \U$ and $V_B,V^B\in \V$.
\end{itemize}
We denote the subcategory of all the objects $B$ such that $U_B,V^B\in \U\cap \V$ in the conflations in (b) by $\h$. We call the ideal quotient $\h/(\U\cap \V)$ the heart of $(\U,\V)$. It is an abelian category by \cite[Theorem 3.2]{LN}.
\end{defn}

\begin{defn}
Let $\B',\B''$ be two subcategories of $\B$, let $$\Cone(\B',\B'')=\{ X\in \B \text{ } | \text{ } X \text{ admits a conflation } B'\rightarrowtail B''\twoheadrightarrow X,B'\in \B',B''\in \B''\},$$
$$\CoCone(\B',\B'')=\{ X\in \B \text{ } | \text{ } X \text{ admits a conflation } X\rightarrowtail B'\twoheadrightarrow B'',B'\in \B',B''\in \B''\}.$$

\end{defn}

When $\B$ has enough projecitves and injectives, a rigid subcategory $\C$ which is controvariantly finite and contains $\mathcal P$ induces a cotorsion pair $(\C,\C^{\bot_1})$ (see Lemma \ref{rigidcp} for details), the functor category (see \cite{Au}) $\mod (\C/\mathcal P)$ is equivalent to the heart of $(\C,\C^{\bot_1})$. Let $\D\subset \C$, when we consider the mutation of $\C$: $\C'=\CoCone(\D,\C)\cap \D^{\bot_1}$ (see \cite[Definition 2.5]{IY}) that induces a cotorsion pair $(\C',\C'^{\bot_1})$ where $\C'$ is also rigid, we can investigate the relation of two functor categories $\mod (\C/\mathcal P)$ and $\mod (\C'/\mathcal P)$ by studying the hearts. We know if $\C^{\bot_1}=\C'^{\bot_1}$, the hearts are equivalent \cite[Proposition 3.12]{LN}. But here it is obvious not the case, since $\C^{\bot_1}=\C'^{\bot_1}$ implise $\C=\C'$.

Although these hearts (hence the functor categories) are not equivalent in general, we can consider the localizations of the hearts, since on this level we can find an equivalence. In this article (compare with \cite{MP}), we choose the language of hearts of cotorsion pairs since it simplifies some proofs, it also makes arguments simpler when we deal with general subcategories compare with the ones obtained from objects. We will prove the following theorem. 

\begin{thm}
Let $(\B,\mathbb{E},\mathfrak{s})$ be an extriangulated category with enough projectives and enough injectives. Let $\C,\C',\D$ be rigid subcategories such that $\D\subseteq\C\cap\C'$. Assume we have three pairs of cotorsion pairs $((\C,\C^{\bot_1}),(\C^{\bot_1},\M))$, $(({\D},{\D}^{\bot_1}),({\D}^{\bot_1},{\mathcal N}))$, $((\C',\C'^{\bot_1}),(\C'^{\bot_1},\M'))$ such that $\CoCone({\D},\C)=\CoCone(\C',{\D})$ and $\Cone({\mathcal N},\M)=\Cone(\M',\mathcal N)$. Let
\begin{itemize}
\item[(a)] $\h/\C$ be the heart of $(\C,\C^{\bot_1})$. Denote $(\h\cap {\D}^{\bot_1})/\C$ by $\mathcal A$. Let $\s_{\mathcal A}$ be the class of epimorphisms in $\h/\C$ whose kernel belong to $\mathcal A$.
\item[(b)] $\h'/\M'$ be the heart of $(\C'^{\bot_1},\M')$. Denote $(\h' \cap {\D}^{\bot_1})/\M'$ by $\mathcal A'$. Let $\s_{\mathcal A'}$ be the class of monomorphisms in $\h'/\M'$ whose cokernel belong to $\mathcal A'$.
\end{itemize}
Then we have the following equivalences:
$$(\h/\C)_{\s_{\mathcal A}} \simeq \CoCone({\D},\C)/\C' \simeq \Cone({\mathcal N},\M)/\M \simeq (\h'/\M')_{\s_{\mathcal A'}}.$$
\end{thm}

Note that condition $\CoCone({\D},\C)=\CoCone(\C',{\D})$ implies $\C'=\CoCone({\D},\C)\cap \D^{\bot_1}$. This result allows us to study the similar mutations as in \cite{MP} on exact categories, where usually we do not have Serre functors. It generalizes the results by Marsh and Palu (see \cite[Theorem 2.9, 3.2]{MP}) for the following reason:

\begin{rem}
Let $\B$ be a Krull-Schmidt, $k$-linear, Hom-finite triangulated category with suspension functor $\Sigma$. Let $C$ be a rigid object and $D$ is a direct summand of $C$, we have $C=D\oplus R$. Let $\C=\add C$ and $\D=\add D$, we have the following triangle $R^*\to D_0\xrightarrow{f} R\to \Sigma R^*$ where $f$ is a minimal right $\D$-approximation. Let $C'=D\oplus R^*$ and $\C'=\add C'$, we assume $\C'$ is rigid. By \cite[Lemma 2.7]{MP}, we have $\CoCone({\D},\C)=\CoCone(\C',{\D})$. Under the assumptions for $\B$, we have cotorsoin pairs $(\C,\C^{\bot_1})$, $({\D},{\D}^{\bot_1})$ and $(\C',\C'^{\bot_1})$. The module category of the endomorphism algebra $\End_{\B}(C)^{\op}$ (resp. $\End_{\B}(C')^{\op}$) is equivalent to the heart of $(\C,\C^{\bot_1})$ (resp. $(\C'^{\bot_1},\M')$) (see \cite[Proposition 4.15]{LN}). Moreover, if $\B$ has a Serre functor $S$, then we also have cotorsion pairs $(\C^{\bot_1},\Sigma^{-2}S\C)$, $(\D^{\bot_1},\Sigma^{-2}S\D)$ and $(\C'^{\bot_1},\Sigma^{-2}S\C')$. One can check that $\Cone(\Sigma^{-2}S\D,\Sigma^{-2}S\C)=\Cone(\Sigma^{-2}S\C',\Sigma^{-2}S\D)$.
\end{rem}

For the readers who are familiar with extriangulated category, just consider the case when $\B$ is an exact category (or abelian category), then $\EE$ is just $\Ext^1_B$, conflations are just short exact sequences. Everything works in the same way. 
 
In section 2, we introduce necessary backgroud knowledge of cotorsion pairs and prove some lemmas which will be used later. In section 3, we study a localization of the heart of a cotorsion pair related to the mutation, which is given as $(\h/\C)_{\s_{\mathcal A}} \simeq \CoCone(\D,\C)/\C'$ in the main theorem. In fact we show the result in a more general setting: we only need $\C'=\CoCone(\D,\C)\cap \D^{\bot_1}$ without assuming its rigidity. In section 4, we prove our main theorem. In section 5 we discuss some localizations of $\B$ related to our main theorem. In the last section, we give some examples of our theorem.

\subsection*{Acknowledgments}
The author would like to thank Martin Herschend, Osamu Iyama and Bin Zhu for their helpful advices and corrections.

\section{Preliminaries}

Throughout this article, let $(\B,\mathbb{E},\mathfrak{s})$ be an extriangulated category defined in \cite{NP} (see $\S2$ of \cite{NP} for details). Let $\mathcal P$ (resp.  $\mathcal I$) be the subcategory of projectives (resp. injectives). Let $\C\supset \D$ be subcategories of $\B$.

We first recall the following proposition (\cite[Proposition 1.20]{LN}), which (also the dual of it) will be used many times in the article.

\begin{prop}\label{1.20}
Let $A\overset{x}{\rightarrowtail}B\overset{y}{\twoheadrightarrow}C\overset{\delta}{\dashrightarrow}$ be any $\EE$-triangle, let $f\colon A\rightarrow D$ be any morphism, and let $D\overset{d}{\rightarrowtail}E\overset{e}{\twoheadrightarrow}C\overset{f_{\ast}\delta}{\dashrightarrow}$ be any $\EE$-triangle realizing $f_{\ast}\delta$. Then there is a morphism $g$ which gives a morphism of $\EE$-triangles
$$\xymatrix{
A \ar@{ >->}[r]^{x} \ar[d]_f &B \ar@{->>}[r]^{y} \ar[d]^g &C \ar@{=}[d]\ar@{-->}[r]^{\delta}&\\
D \ar@{ >->}[r]_{d} &E \ar@{->>}[r]_{e} &C\ar@{-->}[r]_{f_{\ast}\delta}&
}
$$
and moreover, the sequence $A\overset{\svecv{f}{x}}{\rightarrowtail}D\oplus B\overset{\svech{d}{-g}}{\twoheadrightarrow}E\overset{e^{\ast}\delta}{\dashrightarrow}$ becomes an $\EE$-triangle.
\end{prop}

Although most of the time we will deal with cotorsion pairs, it is still necessary to introduce the folloiwng more general concept used in the proof of our main theorem.


\begin{defn}\label{Cctp}
A pair of cotorsion pairs $\STUV$ on $\B$ is called a \emph{twin cotorsion pair} if $\s\subseteq \U$.

Remark that any cotorsion pair $(\U,\V)$ gives a twin cotorsion pair $((\U,\V),(\U,\V))$. Thus a cotorsion pair can be regarded as a special case of a twin cotorsion pair, satisfying $\s=\U$ and $\T=\V$.
\end{defn}

\begin{rem}\label{C3}
For any cotorsion pair $(\U,\V)$ on $\B$, the following holds.
\begin{itemize}
\item[(a)] An object $B\in \U$ if and only if $\EE(B,\V)=0$.

\item[(b)] An object $B\in \V$ if and only if $\EE(\U,B)=0$.

\item[(c)] Subcategories $\U$ and $\V$ are closed under extension.

\item[(d)] $\mathcal P \subseteq \U$ and $\mathcal I \subseteq \V$.

\end{itemize}
\end{rem}

\begin{defn}\label{C5}
For any twin cotorsion pair $\STUV$, put $\W=\T\cap \U$.
\begin{itemize}
\item[(a)] Subcategory $\B^+$ is defined to be the full subcategory of $\B$, consisting of objects $B$ which admit conflations $V_B\rightarrowtail W_B\twoheadrightarrow B$ where $W_B\in \W$ and $V_B\in \V$.
\item[(b)] Subcategory $\B^-$ is defined to be the full subcategory of $\B$, consisting of objects $B$ which admit conflations $B\rightarrowtail W^B\twoheadrightarrow S^B$ where $W^B\in \W$ and $S^B\in \s$.
\end{itemize}
\end{defn}

\begin{defn}\label{C6}
Let $\STUV$ be a twin cotorsion pair on $\B$, and write the ideal quotient of $\B$ by $\W$ as $\oB=\B/\W$. For any morphism $f\in \Hom_\B(X,Y)$, we denote its image in $ \Hom_{\oB}(X,Y)$ by $\overline f$.
For any full additive subcategory $\mathcal B_1$ of $\B$ containing $\W$, similarly we put $\overline{\B_1}=\B_1/\W$. This is a full subcategory of $\oB$ consisting of the same objects as $\mathcal B_1$.

Let $\h=\B^+\cap\B^-$. Since $\h\supseteq \W$, we obtain a subcategory $\overline{\h}\subseteq\oB$, which we call the \emph{heart} of the twin cotorsion pair. It is semi-abelian by \cite[Theorem 2.32]{LN}. In particular, the heart of the twin cotorsion pair $((\U,\V),(\U,\V))$ is the heart of $(\U,\V)$.
\end{defn}

Let $\B_1*\B_2=\{B \in \B \text{ }|\text{ } B \text{ admits a conflation } B_1\rightarrowtail B\twoheadrightarrow B_2,B_1\in \B_1,B_2\in \B_2\}$, for a cotorsion pair $(\U,\V)$ and its heart, according to \cite[Theorem 3.2, 3.5, Corollary 3.8]{LN}, we have the following theorem.

\begin{thm}\label{cohomological}
Let $(\U,\V)$ be a cotorsion pair, then its hearts $\overline \h$ is abelian. Moreover, there exists an additive functor $H:\B\to \overline \h$ such that
\begin{itemize}
\item[$\bullet$] $H|_{\h}=\pi|_{\h}$, where $\pi:B\to \overline \B$ is the quotient functor;
\item[$\bullet$] for any object $X\in \B$, $H(X)=0$ if and only if $X\in \add(\U*\V)$. 
\item[$\bullet$] for any conflation
$\xymatrix{A \ar@{ >->}[r]^{f} &B \ar@{->>}[r]^{g} &C}$ in $\B$, the sequence $H(A)\xrightarrow{H(f)} H(B)\xrightarrow{H(g)} H(C)$ is exact in $\overline \h$.
\end{itemize}

\end{thm}

We call $H$ the \emph{cohomological functor} associated with $(\U,\V)$.

A subcategory $\B'$ is called \emph{contravariantly finite} if any object in $\B$ admits a right $\B'$-approximation. Moreover, it is called \emph{fully} contravariantly finite if any object in $\B$ admits a right $\B'$-approximation which is also a deflation. 
Dually we can define (\emph{faithfully}) \emph{covariantly finite} subcategory.

\begin{lem}\label{rigidcp}
If $\C$ is rigid, closed under direct summands, fully contravariantly finite and $\B$ has enough injectives, then $(\C,\C^{\bot_1})$ is a cotorsion pair.
\end{lem}

\begin{proof}
Since $\B$ has enough injectives, any object $A\in\B$ admits a conflation $\xymatrix{A \ar@{ >->}[r] &I \ar@{->>}[r] &B}$ where $I$ is injective. Since $\C$ is fully contravariantly finite, object $B$ admits a conflation $\xymatrix{B_1 \ar@{ >->}[r] &C_0 \ar@{->>}[r]^{f_0} &B}$ where $f_0$ is a right $\C$-approximation. The rigidity of $\C$ implies $B_1\in \C^{\bot_1}$. We have the following commutative diagram
$$\xymatrix{
&A \ar@{ >->}[d] \ar@{=}[r] &A \ar@{ >->}[d]\\
B_1 \ar@{=}[d] \ar@{ >->}[r] &X \ar@{->>}[r] \ar@{->>}[d] &I \ar@{->>}[d]\\
B_1 \ar@{ >->}[r] &C_0 \ar@{->>}[r]^{f_0} &B
}
$$
where $X\in \C^{\bot_1}$. Hence by definition, the pair $(\C,\C^{\bot_1})$ is a cotorsion pair.
\end{proof}

The following lemma will be used later.

\begin{lem}\label{exactsq}
Let $(\U,\V)$ be a cotorsion pair and $\overline \h$ be its heart. If we have a short exact sequence $0\to A\xrightarrow{\overline f} B\xrightarrow{\overline g} C\to 0$ in $\overline \h$, then we have a conflation $\xymatrix{A' \ar@{ >->}[r]^{f'} &B' \ar@{->>}[r]^{g'} &C}$ where $A',B',C \in \h$ such that its image by applying cohomlogical functor $H$ is isomorphic to $0\to A\xrightarrow{\overline f} B\xrightarrow{\overline g} C\to 0$.
\end{lem}

\begin{proof}
For morphism $g$, we have the following commutative diagram
$$\xymatrix{
V_C \ar@{ >->}[r] \ar@{=}[d] &{K_g}  \ar@{->>}[r]^{k_g} \ar[d] &B \ar[d]^g\\
V_C \ar@{ >->}[r] &{W_C} \ar@{->>}[r]_{w_C} &C}
$$
where $V_C\in \V$ and $W_C\in \W$. Then we obtain a conflation
$$\xymatrix{K_g \ar@{ >->}[r]^-{\svecv{k_g}{-a}} &B\oplus W_C \ar@{->>}[r]^-{\svech{g}{w_C}} &C.}$$
By \cite[Lemma 3.1]{LN}, we have $K_g\in \B^-$; by \cite[Lemma 2.10]{LN}, we have $K_g\in \B^+$. Hence $K_g\in \h$ and we get a short exact sequence $0\rightarrow K_g\xrightarrow{\overline{k_g}} B\xrightarrow{\overline g} C\rightarrow 0$ in $\overline \h$. Hence $K_g\simeq A$ in $\overline \h$.
\end{proof}

In the rest of this article, we always assume $\B$ has enough projectives and enough injectives.

For a subcategory $\B'$, we define $\Omega^0\B'=\B'$ and $\Omega^i\B'$ for $i>0$ inductively by
 $\Omega^i\B'=\CoCone(\mathcal P, \Omega^{i-1}\B')$.
We call $\Omega^{i}\B'$ the {\it $i$-th syzygy} of $\B'$, by this definition we have $\mathcal P\subseteq \Omega^i \B'$, $i>0$. Dually we can define the {\it $i$-th cosyzygy} $\Sigma^i \B'$.


\begin{lem}\label{syzygyapproximation}
If we have a cotorsion pair $(\C,\C^{\bot_1})$, then any object $X$ admits a commutative diagram
$$\xymatrix{
Y_0 \ar@{=}[r] \ar@{ >->}[d]_{g_0} &Y_0 \ar@{ >->}[d]^{y_0}\\
U_0 \ar@{ >->}[r]^{u_0} \ar@{->>}[d]_{f_0} &P_0 \ar@{->>}[r] \ar@{ >->}[d] &C_0 \ar@{=}[d]\\
X \ar@{ >->}[r] &T_0 \ar@{->>}[r] &C_0
}
$$
where $T_0\in \C^{\bot_1}$, $C_0\in \C$, $P_0\in \mathcal P$ and $f_0$ is a right $\Omega \C$-approximation.
\end{lem}

\begin{proof}
If we have a cotorsion pair $(\C,\C^{\bot_1})$, then any object $X$ admits a conflation $\xymatrix{X \ar@{ >->}[r] &T_0 \ar@{->>}[r] &C_0}$ where $T_0\in \C^{\bot_1}$, $C_0\in \C$. Since $\B$ has enough projectives, object $T_0$ admits a conflation $\xymatrix{Y_0 \ar@{ >->}[r]^{y_0} &P_0 \ar@{->>}[r] &T_0}$ where $P_0\in \mathcal P$. Hence by the axiom of extriangulated category, we get a commutative diagram
$$\xymatrix{
Y_0 \ar@{=}[r] \ar@{ >->}[d]_{g_0} &Y_0 \ar@{ >->}[d]^{y_0}\\
U_0 \ar@{ >->}[r]^{u_0} \ar@{->>}[d]_{f_0} &P_0 \ar@{->>}[r] \ar@{ >->}[d]^{p_0} &C_0 \ar@{=}[d]\\
X \ar@{ >->}[r]_t &T_0 \ar@{->>}[r] &C_0
}
$$
which gives rise to a conflation $\xymatrix{U_0\ar@{ >->}[r]^-{\svecv{u_0}{f_0}} &P_0\oplus X \ar@{->>}[r]^-{\svech{-p_0}{t}} &T_0}$. Let $f:U\to X$ be any morphism such that $U$ admits a conflation $\xymatrix{U \ar@{ >->}[r]^u &P \ar@{->>}[r] &C}$ where $P\in\mathcal P$ and $C\in\C$. then we get the following commutative diagram
$$\xymatrix{
&U \ar@{ >->}[r]^u \ar[d]_{\svecv{0}{f}} &P \ar@{->>}[r] \ar@{.>}[d]^p &C\\
U_0\ar@{ >->}[r]_-{\svecv{u_0}{f_0}} &P_0\oplus X \ar@{->>}[r]_-{\svech{-p_0}{t}} &T_0
}
$$
since $\EE(C,T_0)=0$. Object $P$ is projective, hence there is a morphism $P\xrightarrow{\svecv{-a}{b}} P_0\oplus X$ such that $\svech{-p_0}{t}\svecv{-a}{b}=p$ and a morphism $c:P\to U_0$ such that $f_0c=b$. Since $\svech{-p_0}{t}(\svecv{0}{f}-\svecv{-au}{bu})=0$, there exists a morphism $d:U\to U_0$ such that $f=f_0d+bu=f_0(d+cu)$. Hence $f_0$ is a right $\Omega C$-approximation.
\end{proof}

\begin{rem}\label{factorsthroughP}
In the lemma above, if $\EE(T_0,B)=0$ for an object $B$, then $\Hom_{\B}(g_0,B)$ is surjective.
\end{rem}

\begin{lem}\label{syzygyepi}
If we have a cotorsion pair $(\C,\C^{\bot_1})$ where $\C$ is rigid, let $g:A\twoheadrightarrow B$ be a deflation in $\h$ such that $\overline g$ is an epimorphsim where $A,B\in \h$, then $\Hom_B(X,g)$ is surjective whenever $X\in \Omega \C$.
\end{lem}

\begin{proof}
If $g:A\twoheadrightarrow B$ be a deflation in $\h$ such that $\overline g$ is an epimorphsim, then $g$ admits the following commutative diagram:
$$\xymatrix{
A \ar@{ >->}[r]^{a_0} \ar[d]_g &C^0 \ar[d]^{c_0} \ar@{->>}[r] &C^1 \ar@{=}[d]\\
B \ar@{ >->}[r]_h &C_g \ar@{->>}[r] &C^1  
}
$$
where $C^0,C^1,C_g\in \C$. By Proposition \ref{1.20}, we have a conflation $\xymatrix{A \ar@{ >->}[r]^-{\svecv{a_0}{g}} &C_0\oplus B \ar@{->>}[r]^-{\svech{-c_0}{h}} &C_g}$. Let $f:X\to B$ be any morphism such that $X$ admits a conflation $\xymatrix{X \ar@{ >->}[r]^u &P \ar@{->>}[r] &C}$ where $P\in\mathcal P$ and $C\in\C$, then we get the following commutative diagram
$$\xymatrix{
&X \ar@{ >->}[r]^u \ar[d]_{\svecv{0}{f}} &P \ar@{->>}[r] \ar@{.>}[d]^p &C\\
A\ar@{ >->}[r]_-{\svecv{a_0}{g}} &C_0\oplus B \ar@{->>}[r]_-{\svech{-c_0}{h}} &C_g
}
$$
since $\C$ is rigid. Object $P$ is projective, hence there is a morphism $P\xrightarrow{\svecv{-a}{b}} C_0\oplus B$ such that $\svech{-c_0}{h}\svecv{-a}{b}=p$ and a morphism $c:P\to A$ such that $gc=b$. Since $\svech{-c_0}{h}(\svecv{0}{f}-\svecv{-au}{bu})=0$, there exists a morphism $d:X\to A$ such that $f=gd+bu=g(d+cu)$. Hence $\Hom_{\B}(X,g)$ is surjective.

\end{proof}

\section{Localization of hearts}

\begin{defn}
Subcategory $\C$ satisfies condition (RCP) if $\mathcal P\subset \C$, $\C$ is rigid, contravariantly finite, closed under direct summands.

\end{defn}

("RCP" means rigid cotorsoin pair)

By Lemma \ref{rigidcp}, if  $\C$ satisfies condition (RCP) , $(\C,\C^{\bot_1})$ is a cotorsion pair. 

From now on, let $\D\subset\C$ be subcategories satisfying (RCP). Let $\U=\Omega\C$ and $\V=\Omega {\D}$.


Since $(\C,\C^{\bot_1})$ is a cotorsion pair, we have a subcategory $\h$ according to the definition of the heart such that the heart of $(\C,\C^{\bot_1})$ is $\h/\C=:\overline \h$. In this case $\h=\CoCone(\C,\C)$. 

Let $\h_\D=\CoCone(\D,\C)$, then $\U\subseteq \h_\D\subseteq \h$.

\begin{prop}\label{lem2.2}
Any object $X$ admits a conflation $\xymatrix{Z\ar@{ >->}[r] &Y\ar@{->>}[r]^{f} &X}$ where $f$ is a right ${\h_\D}$-approximation and $Z\in {\D}^{\bot_1}$. Moreover, morphism  $x':X\to X'$ factors through $\C^{\bot_1}$ if $x'f$ factors through $\C^{\bot_1}$.
\end{prop}

\begin{proof}
By Lemma \ref{syzygyapproximation}, any object $X$ admits the following commutative diagram
$$\xymatrix{
Y_0 \ar@{=}[r] \ar@{ >->}[d]_{g_0} &Y_0 \ar@{ >->}[d]^{y_0}\\
U_0 \ar@{ >->}[r]^{u_0} \ar@{->>}[d]_{f_0} &P_0 \ar@{->>}[r] \ar@{ >->}[d] &C_0 \ar@{=}[d]\\
X \ar@{ >->}[r] &T_0 \ar@{->>}[r] &C_0
}
$$
where $T_0\in \C^{\bot_1}$, $C_0\in \C$, $P_0\in \mathcal P$ and $f_0$ is a right $\U$-approximation. Object $Y_0$ also admits a conflation $\xymatrix{Y_1 \ar@{ >->}[r]^{g_1} &V_1 \ar@{->>}[r]^{f_1} &Y_0}$ where $f_1$ is a right $\V$-approximation. Object $V_1$ admits a conflation $\xymatrix{V_1 \ar@{ >->}[r]^{h_1} &P_1 \ar@{->>}[r] &D_1}$ where $P_1\in \mathcal P_1$ and $D_1\in \D$. Thus we have the following commutative diagram
$$\xymatrix{
Y_1 \ar@{ >->}[r]^{g_1} \ar@{=}[d] &V_1 \ar@{->>}[r]^{f_1} \ar@{ >->}[d]^{h_1} &Y_0 \ar@{ >->}[r]^{g_0} \ar@{ >->}[d]^v &U_0 \ar@{->>}[r]^{f_0} \ar@{ >->}[d]^u &X \ar@{=}[d]\\
Y_1 \ar@{ >->}[r] &P_1 \ar@{->>}[r]^p \ar@{->>}[d] &Z \ar@{->>}[d]^z \ar@{ >->}[r]^g &Y \ar@{->>}[r]^f \ar@{->>}[d]^y &X\\
& D_1 \ar@{=}[r] & D_1 \ar@{=}[r] &D_1
}
$$
where $\Hom_{\B}(U,f)$ is surjective for any $U\in \U$ since $f_0$ is a right $\U$-approximation.
We claim $\xymatrix{Z  \ar@{ >->}[r]^g &Y \ar@{->>}[r]^f &X}$ is the conflation we need.\\
Since we have the folloiwng commutative diagram of conflations:
$$\xymatrix{
U_0 \ar@{ >->}[r]^u \ar@{=}[d] &Y \ar@{->>}[r]^y \ar@{.>}[d]^{p_0} &D_1 \ar@{.>}[d]\\
U_0 \ar@{ >->}[r]_{u_0} &P_0 \ar[r] &C_0
}
$$
where $P_0\in \mathcal P$ and $C_0\in \C$, by Proposition \ref{1.20}, we get a conflation
$\xymatrix{Y \ar@{ >->}[r]^-{\svecv{y}{p_0}} &D_1\oplus P_0 \ar@{->>}[r] &C_0}$ which implies $Y\in \h_{\D}$.\\
For any subcategory $\B_1$, let $\underline \B_1^{\bot}=\{B\in\B \text{ }| \text{ } \Hom_{\B/{\mathcal P}}(\B_1,B)=0 \}$. Now we check that $Z\in \D^{\bot_1}$. It is enough to show $Z\in \underline \V^{\bot}$ since $\D^{\bot_1}=\underline \V^{\bot}$. Let $V$ be an object in $\V$ that admits a conflation $\xymatrix{V \ar@{ >->}[r]^{b} &P \ar@{->>}[r] &D}$ where $P\in \mathcal P$ and $D\in \D$. Let $a:V\to Z$ be any morphism, since $\D$ is rigid, morphism $za$ factors through $b$. We have the following diagram
$$\xymatrix{
&V \ar@{ >->}[r]^b \ar[d]_a \ar@{.>}[dl]_e &P \ar[d]^c \ar@{.>}[dl]^d \ar@{->>}[r] &D\\
Y_0 \ar@{ >->}[r]_v &Z \ar@{->>}[r]_z &D_1
}
$$
where $za=cb$. Since $P$ is projective, there exists a morphism $d:P\to Z$ such that $c=zd$. Hence $z(a-db)=0$ and there exists a morphism $e:V\to Y_0$ such that $ve=a-db$. Since $f_1$ is a right $\V$-approximation, there exists a morphism $h:V\to V_1$ such that $f_1h=e$. We get $vf_1h=ph_1h=a-db$, then $a$ factors through $\mathcal P$, which implies $Z\in \underline \V^{\bot}$. 

Let $Y'\in \h_\D$, then it admits the following commutative diagram
$$\xymatrix{
V_{Y'} \ar@{=}[r] \ar@{ >->}[d]_{v'} &V_{Y'} \ar@{ >->}[d]\\
U_{Y'} \ar@{ >->}[r] \ar@{->>}[d] &P_{Y'} \ar@{->>}[r] \ar@{ >->}[d] &C_{Y'} \ar@{=}[d]\\
Y' \ar@{ >->}[r] &D_{Y'} \ar@{->>}[r] &C_{Y'}
}
$$
where $D_{Y'}\in \D$, $C_{Y'}\in \C$ and $P_{Y'}\in \mathcal P$. Let $x:Y'\to X$ be a morphism, then we have a commutative diagram of conflations
$$\xymatrix{
V_{Y'} \ar@{ >->}[r]^{v'} \ar@{.>}[d]_{z'} &U_{Y'} \ar@{.>}[d] \ar@{->>}[r] &Y' \ar[d]^{x} \\
Z  \ar@{ >->}[r]^g &Y \ar@{->>}[r]^f &X.
}
$$
Since $Z\in \D^{\bot_1}=\underline \V^{\bot}$, morphism $z'$ factor through $\mathcal P$. By Remark \ref{factorsthroughP}, morphism $z'$ factors through $v'$, which implies $x$ factors through $f$. Hence $f$ is a right $\h_{\D}$-approximation.

Now we show the "moreover" part. Since we have the following commutative diagram of conflations:
$$\xymatrix{
Y_0 \ar@{ >->}[r]^v \ar@{=}[d] &Z \ar@{->>}[r]^z \ar[d]^{p_0g} &D_1 \ar@{.>}[d]\\
Y_0 \ar@{ >->}[r]_{y_0} &P_0 \ar@{->>}[r] &T_0
}
$$
we get a commutative diagram of conflations:
$$\xymatrix{
Z  \ar@{ >->}[r]^g \ar@{=}[d] &Y \ar@{ >->}[d]^-{\svecv{y}{p_0}} \ar@{->>}[r]^f &X \ar@{.>}[d]^t\\
Z \ar@{ >->}[r]_-{\svecv{z}{p_0g}} &D_1\oplus P_0 \ar@{->>}[r] &T_0
}$$
For convenience, we denote the above diagram by 
$$\xymatrix{
Z  \ar@{ >->}[r]^g \ar@{=}[d] &Y \ar@{ >->}[d]^-{y'} \ar@{->>}[r]^f &X \ar[d]^t\\
Z \ar@{ >->}[r]_-{z'} &D_1' \ar@{->>}[r] &T_0.
}$$
Now assume there is a morphism $x':X\to X'$ such that $x'f$ admits a commutative diagram
$$\xymatrix{
Y \ar[r]^f \ar[d]_l &X \ar[d]^{x'}\\
K \ar[r]_k &X'
}
$$
where $K\in \C^{\bot_1}$, then there exists a morpshism $n:D_1'\to K$ such that $l=ny'$  since $\EE(C_0,K)=0$. Hence there is a morphism $t':T_0\to X'$ such that $x'=t't$, which implies $x$ itself factors through $\C^{\bot_1}$.
\end{proof}

The following corollary is an immediate conclusion of Proposition \ref{lem2.2}.

\begin{cor}\label{heartepi}
In the propositon above, if $X\in\h$, then $\overline f$ is an epimorphism in $\overline \h$.
\end{cor}

We also have the following useful corollary.

\begin{cor}\label{cap}
A morphism $g:A\to B$ in $\h_\D$ factors through $\D^{\bot_1}$ only if it factors through $\D^{\bot_1}\cap \h_\D$.
\end{cor}

\begin{proof}
If $f:A\to B$ in $\h_\D$ factors through an object $X\in \D^{\bot_1}$, by Proposition \ref{lem2.2}, there is a conflation $\xymatrix{Z\ar@{ >->}[r] &Y\ar@{->>}[r]^{f} &X}$ where $f$ is a right ${\h_\D}$-approximation and $Z\in \D^{\bot_1}$. Hence we have the following commutative diagram
$$\xymatrix{
&&&A \ar[dl] \ar@{.>}[dll] \ar[dd]^g\\
Z\ar@{ >->}[r] &Y\ar@{->>}[r]_{f} &X \ar[dr]\\
&&&B
}
$$
where $Y\in \D^{\bot_1}$ since $Z\in \D^{\bot_1}$.
\end{proof}





\begin{defn}
Let $H:\B\to \overline \h$ be the cohomological functor as in Theorem \ref{cohomological}. Denote $H({\D}^{\bot_1})$ by $\mathcal A$. Let $\s_{\mathcal A}$ be the class of epimorphisms $\overline f$ whose kernel belong to $\mathcal A$.
\end{defn}

\begin{lem}
We have $\mathcal A=(\h\cap {\D}^{\bot_1})/\C$. 
\end{lem}

\begin{proof}
Since $H|_{\h}=\pi|_{\h}$, the image of $\h\cap {\D}^{\bot_1}$ lies in $\mathcal A$, hence $(\h\cap {\D}^{\bot_1})/\C\subseteq \mathcal A$.\\
Let $B\in \B$, it admits a commutative diagram
$$\xymatrix{
V_0 \ar@{ >->}[d] \ar@{=}[r] &{V_0} \ar@{ >->}[d]\\
B^- \ar@{->>}[d] \ar@{ >->}[r] &{C_0} \ar@{->>}[d] \ar@{->>}[r] &{C^B} \ar@{=}[d]\\
B \ar@{ >->}[r] &{V^B} \ar@{->>}[r] &{C^B}}
$$
where $V^B,V_0\in \C^{\bot_1}$, and $C_0$,$C^B\in \C$. Hence $B^-\in \B^-=\h$, by definition in \cite{LN} we get $H(B)=B^-$. If $B\in {\D}^{\bot_1}$, we get $B^-\in {\D}^{\bot_1}$. Hence $\mathcal A=H(\D^{\bot_1})\subseteq (\h\cap {\D}^{\bot_1})/\C$.
\end{proof}




Now let $\h_{\D}\cap \D^{\bot_1}=\C'$, we call $\C'$ the \emph{right} $\D$-\emph{mutation} of $\C$. Note that we do not require $\C'$ to be rigid in this section.


We have a functor $\eta: \h_\D/\D\hookrightarrow \h/\D\twoheadrightarrow \overline \h$, let $F$ be composition of functor $\eta$ and the localization functor $L_{\s_{\mathcal A}}:\overline \h\to (\overline \h)_{\s_{\mathcal A}}$. By definition we have $F(\C')=0$ in $(\overline \h)_{\s_{\mathcal A}}$. Hence we have the following commutative diagram
$$\xymatrix{
\h_\D/\D \ar[r]^-{\eta} \ar@{->>}[d]_{\pi'} &\overline \h \ar[d]^-{L_{\s_{\mathcal A}}} \\
\h_\D/\C' \ar[r]^-{F'} &(\overline \h)_{\s_{\mathcal A}}
}
$$
where $\pi'$ is the quotient functor. For convinience, we still denote the morphisms in $\h_\D/\D$ by $\overline f$ (where $f$ is the morphism in $\h_\D$) since $f$ factors through $\D$ if and only if it facttors through $\C$. We will show the following theorem, which is a generalization of the first part of \cite[Theorem 3.2]{MP}.

\begin{thm}\label{thm3.2}
The functor $F':\h_\D/\C'\to (\overline \h)_{\s_{\mathcal A}}$ is an equivalence.
\end{thm}

We shall prove it by several steps.

\begin{lem}\label{lem3.6}
Category $\mathcal A$ is closed under taking epimorphisms.
\end{lem}

\begin{proof}
Let $\overline f:Y\to X$ be an epimorphism in $\overline \h$ such that $Y\in {\D}^{\bot_1}$, we show that $X\in {\D}^{\bot_1}$. By \cite[Corollary 2.26]{LN}, we have the following commutative diagram
$$\xymatrix{
Y \ar[r] \ar[d]_{f} &{C^0} \ar[r] \ar[d] &{C^1} \ar@{=}[d]  \\
X \ar[r] &{C_f} \ar[r] &{C^1}
}
$$
where $C^0,C^1,C_f\in \C$. Hence we get a conflation $Y \rightarrowtail X\oplus C^0\twoheadrightarrow C_f$ which implies $X\oplus C^0\in {\D}^{\bot_1}$. Since ${\D}^{\bot_1}$ is closed under direct summands, we have $X\in {\D}^{\bot_1}$.
\end{proof}

By similar method we can show that if $\overline f:Y\to X$ is a monomorphism in $\overline \h$ such that $X\in {\D}^{\bot_1}$, then $Y\in {\D}^{\bot_1}$.

\begin{prop}\label{Prop3.7}
Functor $F'$ is dense.
\end{prop}

\begin{proof}
It is enough to show $F$ is dense.\\
By Proposition \ref{lem2.2}, any object $X\in\h$ admits a conflation $\xymatrix{Z\ar@{ >->}[r]^g &Y\ar@{->>}[r]^{f} &X}$ where $f$ is a right $\h_\D$-approximation and $Z\in {\D}^{\bot_1}$. By Corollary \ref{heartepi}, morphism $\overline f$ is an epimorphism in $\overline \h$. By \cite[Lemma 3.1]{LN}, we get $Z\in \h$. Then we have an exact sequence $Z\xrightarrow{\overline g} Y \xrightarrow{\overline f} X\to 0$, there is an epimorphism from $Z$ to the kernel of $\overline f$. By Lemma \ref{lem3.6}, the kernel of $\overline f$ is in ${\D}^{\bot_1}$. Hence $Y\simeq X$ in $(\overline \h)_{\s_{\mathcal A}}$.
\end{proof}

\begin{prop}\label{full2}
Functor $F'$ is full.
\end{prop}

\begin{proof}
It is enough to show $F$ is full.\\
Consider a morphism $\alpha:X_1\to Y_2$ in $(\overline \h)_{\s_{\mathcal A}}$ having the form $X_1\xrightarrow{\overline x} X_2 \xrightarrow {\overline f^{-1}} Y_2$ where $X_1,Y_2\in \h_\D$, by definition we have a short exact sequence $0 \to Z_2\xrightarrow{\overline g} Y_2 \xrightarrow{\overline f} X_2\to 0$ in $\overline \h$ where $Z_2\in {\D}^{\bot_1}$. By Lemma \ref{exactsq}, we have a conflation $\xymatrix{Z_2' \ar@{ >->}[r]^{g'} &Y_2' \ar@{->>}[r]^{f'} &X_2}$ in $\h$ such that $\overline f =\overline f'$, $Y_2'=Y_2$ in $\overline \h$ and its image is isomorphic to the short exact sequence. Hence $Z_2'\in {\D}^{\bot_1}$. By Lemma \ref{syzygyepi}, morphism $\overline f'$ is an epimorphism implies $\Hom_{\B}(U,f')$ is surjective for any $U\in \U$. Since $X_1$ admits a conflation $\xymatrix{ V_1 \ar@{ >->}[r]^u &U_0 \ar@{->>}[r] &X_1}$ where $U_0\in \U$ and ${V_1}\in \V$, we have the following commutative diagram
$$\xymatrix{
{V_1} \ar[d]_z \ar@{ >->}[r]^u &U_0 \ar[d]^y \ar@{->>}[r] &X_1 \ar[d]^x\\
Z_2' \ar@{ >->}[r]^{g'} &Y_2' \ar@{->>}[r]^{f'} &X_2
}
$$
where $z$ factors through $\mathcal P$. Hence by Remark \ref{factorsthroughP}, morphism $z$ factors through $u$, which implies there is a morphism $x':X_1\to Y_2'$ such that $x=f'x'$. Hence we have $\overline x'=\overline f^{-1}\overline x$ in $(\overline \h)_{\s_{\mathcal A}}$, which shows $F$ is full.
\end{proof}

\begin{prop}\label{faithful}
Functor $F'$ is faithful.
\end{prop}

\begin{proof}
Since $(\D,\D^{\bot_1})$ is also a cotorsion pair, we denote its heart by $\h_0/\D$ and the associated cohomological functor by $H_0$. Since $H_0(\C^{\bot_1})=0$, we have the following commutative diagram.
$$\xymatrix{
\B \ar[dr]_{H_0} \ar[rr]^{H} &&\overline \h. \ar@{.>}[dl]^{K_0}\\
&\h_0/\D
}$$
Now let $X,Y\in \h_\D$ and $Y\xrightarrow{\overline f}X$ be a morphism in $\s_{\mathcal A}$, then we have a conflation $\xymatrix{Z' \ar@{ >->}[r]^{g'} &Y' \ar@{->>}[r]^{f'} &X}$ in $\h$ such that $\overline f =\overline f'$, $Y'=Y$ in $\overline \h$ and $Z'\in \D^{\bot_1}\cap\h$. Then $H_0(Z')=0$, which means $H_0(f')$ is a monomorphism. Moreover, since $\overline f'=\overline f$ is an epimorphism, we have the following commutative diagram
$$\xymatrix{
Y' \ar[r]^c \ar[d]_{f} &{C^0} \ar[r] \ar[d] &{C^1} \ar@{=}[d]  \\
X \ar[r] &{C_f} \ar[r] &{C^1}
}
$$
where $C^0,C^1,C_f\in \C$. Then we get a conflation $\xymatrix{Y' \ar@{ >->}[r]^-{\svecv{f}{c}} &X\oplus C^0 \ar@{->>}[r] &C_f}$. By applying $H_0$ to this conflation, we get $H_0(f')$ is an epimorphism. Hence $H_0(f')=K_0(\overline f)$ is an isomorphism. By the universal property of $L_{\s_{\mathcal A}}$, there is a functor $J:(\overline \h)_{\s_{\mathcal A}}\to \h_0/\D$ such that $JL_{\s_{\mathcal A}}=K_0$.\\
Now assume we have $X,Y\in \h_\D$ and $u,v\in \Hom_\B(X,Y)$ such that $F(\overline u)=F(\overline v)$, then $H_0(u)=K_0(\overline u)=JL_{\s_{\mathcal A}}(\overline u)=JF(\overline u)=JF(\overline v)=JL_{\s_{\mathcal A}}(\overline v)=K_0(\overline v)=H_0(v)$. Morphism $K_0(\overline a)=0$ if and only if $a$ factors through $\D^{\bot_1}$ by  \cite[Proposition 2.22]{LN}， then $u-v$ factors though $\D^{\bot_1}$. By Corollary \ref{cap}, it factors through $\C'$. Hence $\pi'(\overline u)=\pi'(\overline v+\overline u-\overline v)=\pi'(\overline v)$, which shows $F'$ is faithful.
\end{proof}

Since $F'$ is fully-faithful and dense, it is an equivalence. Now we finished the proof of Theorem \ref{thm3.2}.

Let $ \mathcal N\subset \M'$ be rigid categories such that both $\M'$ and $\mathcal N$ are covariantly finite, closed under direct summands and contain $\mathcal I$. 

Since $({^{\bot_1}}\M',\M')$ is a cotorsion pair, denote $\Cone(\M',\M')$ by $\h'$, the heart of $({^{\bot_1}}\M',\M')$ is $\h'/\M'=:\overline \h'$. 

Let $\h'_{\mathcal N}=\Cone(\M',\mathcal N)$, $H':\B\to \overline \h'$ be the associated cohomological functor. Denote $\h'_{\mathcal N}\cap {{^{\bot_1}}{\mathcal N}}$ by $\M$.

Denote $H'({^{\bot_1}}{\mathcal N})$ by $\mathcal A'$. Let $\s_{\mathcal A'}$ be the class of monomorphisms in $\overline \h'$ whose cokernels belong to $\mathcal A'$.

We have a functor $\eta': \h'_{\mathcal N}/\mathcal N\hookrightarrow  \h'/\mathcal N\twoheadrightarrow \overline \h'$. Let $G$ be composition of functor $\eta$ and the localization functor $L_{\s_{\mathcal A'}}:\overline \h'\to (\overline \h')_{\s_{\mathcal A'}}$. Since $H'(\M)\subseteq \mathcal A'$, we have $G(\M)=0$ in $(\overline \h')_{\s_{\mathcal A'}}$. Hence we have the following commutative diagram.
$$\xymatrix{
\h'_{\mathcal N}/\mathcal N \ar[r]^-{\eta'} \ar@{->>}[d]_{\pi''} &\overline \h' \ar[d]^-{L_{\s_{\mathcal A'}}} \\
\h'_{\mathcal N}/\M \ar[r]^-{{G'}} &(\overline \h')_{\s_{\mathcal A'}}
}
$$
where $\pi''$ is the quotient functor.

\begin{thm}\label{Dthm3.2}
Functor $G':\h'_{\mathcal N}/\M\to (\overline \h')_{\s_{\mathcal A'}}$ is an equivalence.
\end{thm}

\begin{proof}
This is a dual of Theorem \ref{thm3.2}.
\end{proof}

\section{Pseudo-Morita equivalences}

In this section, we prove the main theorem of this paper: 

\begin{thm}\label{main}
Let $(\B,\mathbb{E},\mathfrak{s})$ be an extriangulated category with enough projectives and enough injectives. Assume we have three twin cotorsion pairs $((\C,\C^{\bot_1}),(\C^{\bot_1},\M))$, $(({\D},{\D}^{\bot_1}),({\D}^{\bot_1},{\mathcal N}))$, $((\C',\C'^{\bot_1}),(\C'^{\bot_1},\M'))$ such that $\D\subseteq\C\cap\C'$, $\h_{\D}:=\CoCone({\D},\C)=\CoCone(\C',{\D})$ and $\Cone({\mathcal N},\M)=\Cone(\M',\mathcal N)=:\h_{\mathcal N}'$. Let
\begin{itemize}
\item[(a)] $\overline \h$ be the heart of $(\C,\C^{\bot_1})$. Denote $(\h\cap {\D}^{\bot_1})/\C$ by $\mathcal A$. Let $\s_{\mathcal A}$ be the class of epimorphisms in $\overline \h$ whose kernel belongs to $\mathcal A$.
\item[(b)] $\overline \h'$ be the heart of $(\C'^{\bot_1},\M')$. Denote $(\h' \cap {\D}^{\bot_1})/\M'$ by $\mathcal A'$. Let $\s_{\mathcal A'}$ be the class of monomorphisms in $\overline \h'$ whose cokernel belongs to $\mathcal A'$.
\end{itemize}
Then we have the following equivalences:
$$(\overline \h)_{\s_{\mathcal A}} \simeq \h_{\D}/\C' \simeq \h'_{\mathcal N}/\M \simeq (\overline \h')_{\s_{\mathcal A'}}.$$
\end{thm}

Inspire by \cite{MP}, we call the equivalence $(\overline \h)_{\s_{\mathcal A}} \simeq (\overline \h')_{\s_{\mathcal A'}}$ \emph{pseudo-Morita equivalence}.

\begin{rem}
If $\C$, $\C'$ and $\D$ are subcategories satisfying condition (RCP) and $D \subseteq \C\cap \C'$, then $\CoCone({\D},\C)=\CoCone(\C',{\D})$ if and only if $\C'=\CoCone({\D},\C)\cap \D^{\bot_1}$.
\end{rem}

\begin{rem}
By \cite[Proposition 3.12, 4.15]{LN}, heart $\overline \h'$ is equivalent to the heart of $(\C',\C'^{\bot_1})$, which is equivalent to $\mod(\C'/\mathcal P)$. We also have $\overline \h\simeq \mod(\C/\mathcal P)$.
\end{rem}

According to the previous results, we only need to show $\h_{\D}/\C' \simeq \h'_{\mathcal N}/\M$.







Since $((\C^{\bot_1},\M),(\D^{\bot_1},\mathcal N))$ is a twin cotorsion pair, by Definition  \ref{C5} we have a subcategory $\B^+$ associated with this twin cotorsion pair. In fact $\B^+=\h'_\mathcal N$. By \cite[Definition 2.21]{LN}, the inclusion functor $i^+\colon\B^+/\M=\h'_{\mathcal N}/\M\hookrightarrow\B/\M$ has a right adjoint functor $\sigma^+$ such that every object $B$ admits the following commutative diagram
$$\xymatrix{
V_B \ar@{ >->}[r] \ar@{=}[d] &U_B \ar@{->>}[r] \ar@{ >->}[d] &B \ar@{ >->}[d]^f\\
V_B \ar@{ >->}[r] &T \ar@{->>}[r] \ar@{->>}[d] &B^+ \ar@{->>}[d]\\
&S \ar@{=}[r] &S
}
$$
where $U_B\in \D^{\bot_1}, V_B\in \mathcal N, T\in \M, S\in \C^{\bot_1}$ and $\sigma^+(B)=B^+$, morphism $f$ is a left $\h'_{\mathcal N}$-approximation. Let $H_0'$ be the cohomological functor associated with the heart of $(\D^{\bot_1},\mathcal N)$, then $H_0'(f)$ is an isomorphism. We call the conflation $\xymatrix{B \ar@{ >->}[r]^f &B^+ \ar@{->>}[r] &S}$ a \emph{reflection conflation} of $B$. For every object $B$, we fix a reflection conflation of it. Then for any morphism $x\colon B\rightarrow C$, we define $\sigma^+(\overline x)$ as the unique image of the morphism which makes the following diagram commute (see \cite[Definition 2.21]{LN}).
$$\xymatrix{
B \ar[r]^{x} \ar[d]_{{f}} &C \ar[d]^{ {g}}\\
B^+ \ar@{-->}[r]_{x^+} &C^+.
}
$$
By \cite[Proposition 2.22]{LN}, morphism $\sigma^+(\overline f)=0$ if and only if $f$ factors through $\D^{\bot_1}$. Since $\C'\subseteq \D^{\bot_1}$, we have the following commutative diagram
$$\xymatrix{
\B \ar[d]_{\pi_\M} \ar[r]^{\pi_{\C'}} &\B/\C' \ar@{.>}[d]^{\overline \sigma^+}\\
\B/\M \ar[r]_{\sigma^+} &\h'_{\mathcal N}/\M
}
$$
where $\pi_{\C'}$ and $\pi_{\M}$ are quotient functors. Hence we have a functor $K:\h_\D/\C'\hookrightarrow \B/\C' \xrightarrow{\overline \sigma^+} \h'_{\mathcal N}/\M$.\\
On the other hand, since we have a twin cotorsion pair $((\D,\D^{\bot_1}),(\C',{\C'}^{\bot_1}))$, by the definition we have a subcategory $\B^-$ associated with this twin cotorsion pair. By \cite[Definition 2.21]{LN}, the inclusion $i^-\colon\B^-/\C'=\h_\D/\C'\hookrightarrow\B/\C'$ has a left adjoint functor $\sigma^-$ such that every object $B$ admits a conflation $\xymatrix{V \ar@{ >->}[r] &B^- \ar@{->>}[r]^{f'} &B}$ where $\sigma^-(B)=B^-$, $V\in {\C'}^{\bot_1}$ and $f$ is a right $\h_\D$-approximation, we call this conflation a \emph{coreflection conflation} of $B$. Then dually we have a functor $K':\h'_{\mathcal N}/\M \to \h_{\D}/\C'$. We will prove that $K'K\simeq \id_{\h_{\D}/\C'}$, and dually we can get $KK'\simeq \id_{\h'_{\mathcal N}/\M}$, which shows $\h'_{\mathcal N}/\M$ and ${\h_\D}/\C'$ are equivalent.

\begin{prop}
There is a natural isomorphism between functors $K'K$ and $\id_{\h_{\D}/\C'}$.
\end{prop}

\begin{proof}
Let $B\in \h_\D$, we have a reflection conflation $\xymatrix{B \ar@{ >->}[r]^{f} &B^+ \ar@{->>}[r]^g &S}$, let $\xymatrix{V \ar@{ >->}[r] &B' \ar@{->>}[r]^{f'} &B^+}$ be a coreflection conflation of $B^+$. By the proof of Proposition \ref{Prop3.7}, we have $L_{\s_{\mathcal A}}H(f'):B'\xrightarrow{\simeq} H(B^+)$ in $(\overline \h)_{\s_{\mathcal A}}$. Note that $K'K(B)=B'$ in $\h_{\D}/\C'$.\\
We have the following commutative diagram
$$\xymatrix{
\Omega S \ar[d]_s \ar@{ >->}[r]^{q_S} &P_S \ar[d] \ar@{->>}[r]^{p_S} &S \ar@{=}[d]\\
B \ar@{ >->}[r]^{f} &B^+ \ar@{->>}[r]^g &S
}
$$
which induces a conflation $\xymatrix{\Omega S \ar@{ >->}[r]^-{\svecv{s}{q_S}} &B\oplus P_S \ar@{->>}[r]^-{\svech{f}{-p'}} &B^+}$. Since $H_0'(f)$ is an isomorphism, we have $H_0'(s)=0$, which implies $s$ factors through $\D^{\bot_1}$. Object $\Omega S$ admits a conflation $\Omega S\rightarrowtail {T} \twoheadrightarrow D_1$ where ${T} \in \D^{\bot_1}$ and $D_1\in \D$, hence $s$ factors through $T$. We have the following commutative diagram in $\overline \h$ by applying $H$:
$$\xymatrix{
H(\Omega S) \ar[rr]^-{H(s)} \ar@{->>}[dr] \ar@{->>}[ddr] && B \ar[r]^-{H(f)} &H(B^+) \ar[r] &0.\\
&H(T) \ar@{.>}[d]^{\overline g} \ar[ur]\\
&\Ker(H(f)) \ar@{ >->}[uur]
}
$$
It implies $\overline g$ is an epimorphism. Hence by Lemma \ref{lem3.6}, $\Ker(H(f))\in \mathcal A$. This implies $L_{\s_{\mathcal A}}H(f): B\xrightarrow{\simeq} H(B^+)$ in $(\overline \h)_{\s_{\mathcal A}}$.\\
Let $x:B_0\to B_1$ be a morphism in $\h_\D$, denote its image in $\h_{\mathcal D}/\C'$ by ${\underline x}$, then we have the following commutative diagram:
$$\xymatrix{
B_0 \ar[r]^{f_0} \ar[d]_x &B_0^+ \ar[d]^{x^+} &B_0' \ar[l]_{f_0'} \ar[d]^y\\
B_1 \ar[r]_{f_1} &{B_1}^+  &B_1' \ar[l]^{f_1'}
}
$$
where the image of $y$ in ${\h_D}/\C'$ is $K'K({\underline x})$. Since $L_{\s_{\mathcal A}}H(f_i)$ and $L_{\s_{\mathcal A}}H(f_i')$ are invertible, $i\in \{0,1\}$, by Proposition \ref{full2} we have isomorphisms ${\underline {b_i}}:B_i\to K'K(B_i)$ such that ${F'}( {\underline {b_i}})=L_{\s_{\mathcal A}}H(f_i')^{-1}L_{\s_{\mathcal A}}H(f_i)$. Then we have the following commutative diagram in $\h_{\mathcal D}/\C'$:
$$\xymatrix{
B_0 \ar[r]^-{{\underline {b_0}}}_-{\simeq} \ar[d]_{{\underline x}}  &K'K(B_0)  \ar[d]^{K'K({\underline x})}\\
B_1 \ar[r]_-{{\underline {b_1}}}^-{\simeq}  &K'K(B_1)
}
$$
Hence $K'K\simeq \id_{\h_{\D}/\C'}$.
\end{proof}

Dually we can show $KK'\simeq \id_{\h'_{\mathcal N}/\M}$. Now we finished the proof of Theorem \ref{main}.


\section{More Localizations}

Let $\D\subset\C$ be subcategories satisfying (RCP), let $\h_{\D}\cap \D^{\bot_1}=\C'$ (the same as in section 3).

Let $f:Y\to X$ be a morphism in $\B$, then it admits the following commutative diagram ($\diamond$):
$$\xymatrix{
\Omega X \ar@{ >->}[r] \ar@{=}[d] &Z_1 \ar[d] \ar@{->>}[r]^{g} &Y \ar[d]^f \ar@{ >->}[r] &I^Y \ar[d] \ar@{->>}[r] &\Sigma Y \ar@{=}[d]\\
\Omega X \ar@{ >->}[r] &P_X \ar@{->>}[r] &X \ar@{ >->}[r]^h &Z_2 \ar@{->>}[r] &\Sigma Y
}
$$
where $P_X\in \mathcal P$ and $I^{Y}\in \mathcal I$. Let $\widetilde{\R_1}$ be the class of morphisms $f$ there is a commutative diagram ($\diamond$) where $g$ factors through $\D^{\bot_1}$ and $h$ factors through ${\C}^{\bot_1}$. Let ${\R_1}$ be the class of morphisms $f$ such that there is a commutative diagram ($\diamond$) where $Z_1\in \D^{\bot_1}$ and $h$ factors through ${\C}^{\bot_1}$. Then $\widetilde{\R_1}\supseteq \R_1$. Let $\B_{\R_1}$ (resp. $\B_{\widetilde{\R_1}}$) be the localization of $\B$ at $\R_1$ (resp. $\widetilde{\R_1}$) and $L_{\R_1}$ (resp. $L_{\widetilde{\R_1}}$) be the localization functor. 
If $f\in \widetilde{\R_1}$, object $Z_1$ admits a conflation $Z_1\rightarrowtail {T_1} \twoheadrightarrow \D_1$ where ${T_1} \in \D^{\bot_1}$ and $D_1\in \D$, hence $g$ factors through ${T_1}$. We have the following commutative diagram in $\overline \h$ by applying $H$:
$$\xymatrix@C=0.6cm@R0.6cm{
H(Z_1) \ar[rr]^-{H(g)} \ar@{->>}[dr] \ar@{->>}[ddr] && H(Y) \ar[r]^-{H(f)} &H(X) \ar[r] &0\\
&H({T_1}) \ar@{.>}[d]^{\overline s} \ar[ur]\\
&\Ker(H(f)) \ar@{ >->}[uur]
}
$$
which implies $\overline s$ is an epimorphism. Hence by Lemma \ref{lem3.6}, $\Ker(H(f))\in \mathcal A$. This implies $H(f)\in \s_{\mathcal A}$. Then we have the following commutative diagram.
$$\xymatrix@C=0.6cm@R0.6cm{
\B \ar[dr]^-{L_{{\R_1}}} \ar[dddr]_-{L_{\widetilde{\R_1}}} \ar[rr]^-{L_{\s_{\mathcal A}}H} &&(\overline \h)_{\s_{\mathcal A}}\\
&\B_{{\R_1}} \ar[ur]^{{G_1}} \ar[dd]^J \\
\\
&\B_{\widetilde{\R_1}} \ar[uuur]_{\widetilde{G_1}}
}
$$
We will show that ${G_1}$ is an equivalence. This implies $L_{{\R_1}}$ inverts all the morphism in $\widetilde{\R_1}$, then we have a unique functor $I:\B_{\widetilde{\R_1}}\to \B_{{\R_1}}$ such that $L_{{\R_1}}=L_{\widetilde{\R_1}}I$. Hence $JI=\id$ and $IJ=\id$, which means $\widetilde{G_1}$ is also an equivalence.

\begin{rem}
In the conflation $\xymatrix{Z\ar@{ >->}[r] &Y\ar@{->>}[r]^{f} &X}$ in Proposition \ref{lem2.2}, morphism $f\in \mathcal R_1$.
\end{rem}

The following theorem together with the arguments above generalize \cite[Theorem 3.19]{MP}.



\begin{thm}\label{equivalent2}
Functor ${G_1}:\B_{{\R_1}}\to (\overline \h)_{\s_{\mathcal A}}$ is an equivalence.
\end{thm}

\begin{proof}
Since $H|_{\h}=\pi|_{\h}$, we get ${G_1}$ is dense. We show ${G_1}$ is full.\\
Let $\alpha: {G_1}(X_1)\to {G_1}(X_2)$ be a morphism. By Proposition \ref{lem2.2}, for $i\in \{1,2\}$, $X_i$ admits a conflation $\xymatrix{Z_i\ar@{ >->}[r] &Y_i\ar@{->>}[r]^{f_i} &X_i}$ where $f_i\in {\R_1}$ and $Y_i\in {\h_\D}$, then we have an isomorphism $L_{\s_{\mathcal A}}H(f_i)$ and $L_{\s_{\mathcal A}}H(f_1)\alpha L_{\s_{\mathcal A}}H(f_2)^{-1}:Y_1\to Y_2$. By Proposition \ref{full2}, there exists a morphism $g:Y_1\to Y_2$ such that $L_{\s_{\mathcal A}}H(g)=L_{\s_{\mathcal A}}H(f_1)\alpha L_{\s_{\mathcal A}}H(f_2)^{-1}$. Hence $\alpha=L_{\s_{\mathcal A}}H(f_1)^{-1}L_{\s_{\mathcal A}}H(g)L_{\s_{\mathcal A}}H(f_2)={G_1}(f_1^{-1}gf_2)$.\\
We show ${G_1}$ is faithful. It is enough to check ${G_1}L_{{\R_1}}(a_1)={G_1}L_{{\R_1}}(a_2)$ implies $a_1=a_2$ in $\B_{{\R_1}}$ where $a_1,a_2$ are morphisms from $X_1$ to $X_2$ in $\B$.\\
By the discussion above, we have $b_i:Y_1\to Y_2$ such that $f_2b_i=a_if_1$, $i\in \{1,2\}$. Hence $$L_{\s_{\mathcal A}}H(b_1)=L_{\s_{\mathcal A}}H(f_2)^{-1}L_{\s_{\mathcal A}}H(a_1)L_{\s_{\mathcal A}}H(f_1)=L_{\s_{\mathcal A}}H(f_2)^{-1}L_{\s_{\mathcal A}}H(a_2)L_{\s_{\mathcal A}}H(f_1)=L_{\s_{\mathcal A}}H(b_2)$$
This implies $b_1-b_2$ factors through $\C'\subseteq \D^{\bot_1}$ by Proposition \ref{faithful}.
Hence we have $b_1=b_2+(b_1-b_2)=b_2$ and $a_1=f^{-1}_1b_1f_2=f^{-1}_1b_2f_2=a_2$ in $\B_{{\R_1}}$.
\end{proof}

Now let ${\R_0}$ be the class of morphisms $f$ such that there is a commutative diagram ($\diamond$) where $Z_1\in \C^{\bot_1}$ and $h$ factors through ${\C}^{\bot_1}$, let ${\R_2}$ be the class of morphisms $f$ such that there is a commutative diagram ($\diamond$) where $Z_1\in \D^{\bot_1}$ and $h$ factors through ${\D}^{\bot_1}$. Then $\R_0\subseteq \R_1\subseteq \R_2$. Let $\B_{\R_0}$ (resp. $\B_{\R_2}$) be the localization of $\B$ at $\R_0$ (resp. $\R_2$) and $L_{\R_0}$ (resp. $L_{\R_2}$) be the localization functor. Since $H(f)$ (resp. $H_0(f)$) is an isomorphism if $f\in \R_0$ (resp. $f\in \R_2$) , We have the following commutative diagram:


$$
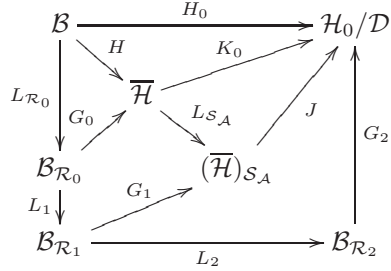
\xymatrix@C=0.4cm@R0.4cm{
\B \ar[dr]^H \ar[dd]_{L_{\R_0}} \ar[rrr]^-{H_0} &&&\h_0/\D\\
&\overline \h \ar[urr]^-{K_0} \ar[dr]^{L_{\s_{\mathcal A}}}\\
\B_{\R_0} \ar[ur]^{G_0} \ar[d]_{L_1} &&(\overline \h)_{\s_{\mathcal A}} \ar[uur]_J\\
\B_{\R_1} \ar[rrr]_{L_2} \ar[urr]^{G_1} &&& \B_{\R_2} \ar[uuu]_{G_2}
}
$$

where $L_1L_{\R_0}=L_{\R_1}$ and $L_2L_1L_{\R_0}=L_{\R_2}$.

\begin{rem}
By the similar method as Proposition \ref{lem2.2}, we can show the following statement: \\
Any object $X$ admits a conflation $\xymatrix{Z\ar@{ >->}[r] &Y\ar@{->>}[r]^{f} &X}$ where $f$ is a right ${\h}$-approximation (resp. $\h_0$-approximation) and $Z\in {\C}^{\bot_1}$ (resp. ${\D}^{\bot_1}$). Moreover, morphism  $x:X\to X'$ factors through $\C^{\bot_1}$ (resp. ${\D}^{\bot_1}$) if $xf$ factors through $\C^{\bot_1}$ (resp. $\D^{\bot_1}$).\\
One can check morphism $f$ in the statement belongs to $\R_0$ (resp. $\R_2$). Then by the similar method as Theorem \ref{equivalent2}, we can show that $G_0$ (resp. $G_2$) is an equivalence.\\
The detail of the proofs is left to the readers.
\end{rem}

\section{Exmaple}

In the last section we give an exmaple of our result in module category.

\begin{exm}\label{ex1}
Let $\Lambda$ be the $k$-algebra given by the quiver
$$
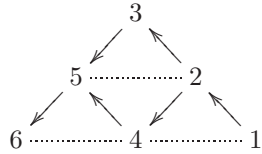
\xymatrix@C=0.4cm@R0.4cm{
&&3 \ar[dl]\\
&5 \ar[dl] \ar@{.}[rr] &&2 \ar[dl] \ar[ul]\\
6 \ar@{.}[rr] &&4 \ar[ul] \ar@{.}[rr] &&1 \ar[ul]}$$
with mesh relations. The AR-quiver of $\B:=\mod\Lambda$ is given by
$$
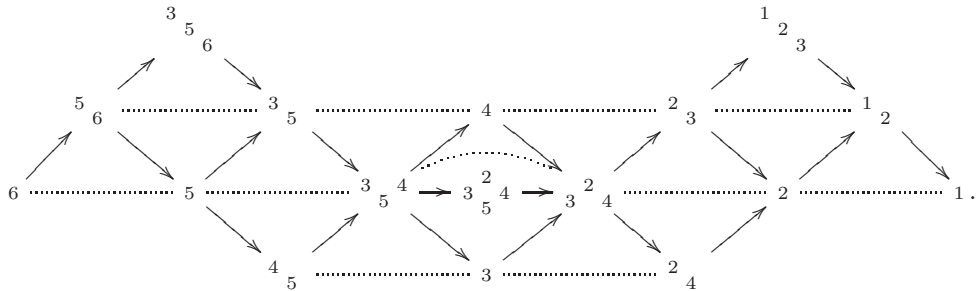
\xymatrix@C=0.4cm@R0.4cm{
&&{\begin{smallmatrix}
3&&\\
&5&\\
&&6
\end{smallmatrix}} \ar[dr] &&&&&&{\begin{smallmatrix}
1&&\\
&2&\\
&&3
\end{smallmatrix}} \ar[dr]\\
&{\begin{smallmatrix}
5&&\\
&6&
\end{smallmatrix}} \ar[ur] \ar@{.}[rr] \ar[dr] &&{\begin{smallmatrix}
3&&\\
&5&
\end{smallmatrix}} \ar@{.}[rr] \ar[dr] &&{\begin{smallmatrix}
4
\end{smallmatrix}} \ar@{.}[rr] \ar[dr] &&{\begin{smallmatrix}
2&&\\
&3&
\end{smallmatrix}} \ar[ur] \ar@{.}[rr] \ar[dr] &&{\begin{smallmatrix}
1&&\\
&2&
\end{smallmatrix}} \ar[dr]\\
{\begin{smallmatrix}
6
\end{smallmatrix}} \ar[ur] \ar@{.}[rr] &&{\begin{smallmatrix}
5
\end{smallmatrix}} \ar[ur] \ar@{.}[rr] \ar[dr] &&{\begin{smallmatrix}
3&&4\\
&5&
\end{smallmatrix}} \ar[ur] \ar[r] \ar[dr] \ar@{.}@/^15pt/[rr] &{\begin{smallmatrix}
&2&\\
3&&4\\
&5&
\end{smallmatrix}} \ar[r] &{\begin{smallmatrix}
&2&\\
3&&4
\end{smallmatrix}} \ar[ur] \ar@{.}[rr] \ar[dr] &&{\begin{smallmatrix}
2
\end{smallmatrix}} \ar[ur] \ar@{.}[rr] &&{\begin{smallmatrix}
1
\end{smallmatrix}}.\\
&&&{\begin{smallmatrix}
4&&\\
&5&
\end{smallmatrix}} \ar[ur] \ar@{.}[rr] &&{\begin{smallmatrix}
3
\end{smallmatrix}} \ar[ur] \ar@{.}[rr] &&{\begin{smallmatrix}
2&&\\
&4&
\end{smallmatrix}} \ar[ur]
}$$
We denote by ``$\circ$" in the AR-quiver the indecomposable objects belong to a subcategory and by ``$\cdot$'' the indecomposable objects do not belong to it.
$$
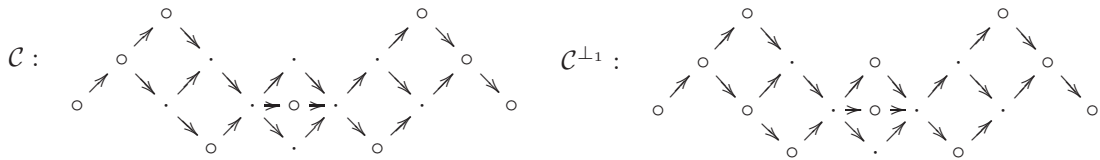
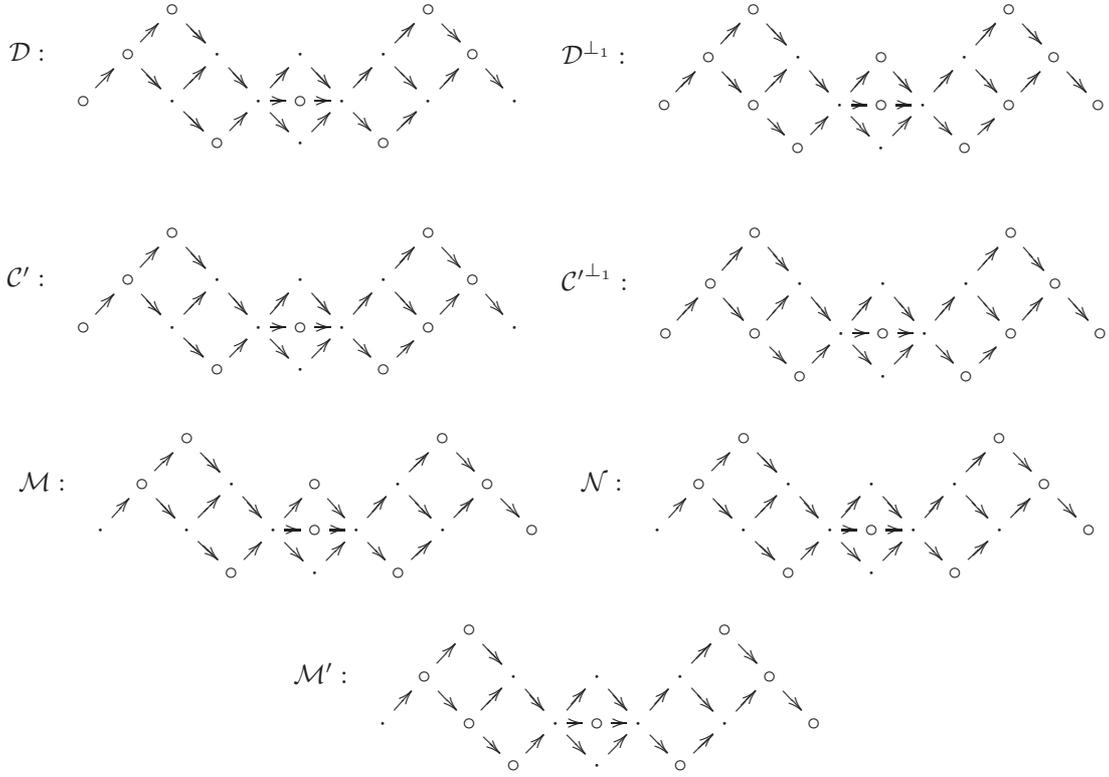
\xymatrix@C=0.2cm@R0.2cm{
&&&\circ \ar[dr] &&&&&&\circ \ar[dr]\\
{\C:} &&\circ \ar[ur]  \ar[dr] &&\cdot  \ar[dr] &&\cdot  \ar[dr] &&\cdot  \ar[ur]  \ar[dr] &&\circ \ar[dr]\\
&\circ \ar[ur]  &&\cdot \ar[ur]  \ar[dr] &&\cdot \ar[ur] \ar[r] \ar[dr] &\circ \ar[r] &\cdot \ar[ur] \ar[dr] &&\cdot \ar[ur] &&\circ\\
&&&&\circ \ar[ur] &&\cdot \ar[ur] &&\circ \ar[ur]
\\} \quad
\xymatrix@C=0.2cm@R0.2cm{
&&&\circ \ar[dr] &&&&&&\circ \ar[dr]\\
{\C^{\bot_1}:} &&\circ \ar[ur]  \ar[dr] &&\cdot  \ar[dr] &&\circ  \ar[dr] &&\cdot  \ar[ur]  \ar[dr] &&\circ \ar[dr]\\
&\circ \ar[ur]  &&\circ \ar[ur]  \ar[dr] &&\cdot \ar[ur] \ar[r] \ar[dr] &\circ \ar[r] &\cdot \ar[ur] \ar[dr] &&\cdot \ar[ur] &&\circ\\
&&&&\circ \ar[ur] &&\cdot \ar[ur] &&\circ \ar[ur]
}$$

$$\xymatrix@C=0.2cm@R0.2cm{
&&&\circ \ar[dr] &&&&&&\circ \ar[dr]\\
{\D:} &&\circ \ar[ur]  \ar[dr] &&\cdot  \ar[dr] &&\cdot  \ar[dr] &&\cdot  \ar[ur]  \ar[dr] &&\circ \ar[dr]\\
&\circ \ar[ur]  &&\cdot \ar[ur]  \ar[dr] &&\cdot \ar[ur] \ar[r] \ar[dr] &\circ \ar[r] &\cdot \ar[ur] \ar[dr] &&\cdot \ar[ur] &&\cdot\\
&&&&\circ \ar[ur] &&\cdot \ar[ur] &&\circ \ar[ur]
\\} \quad
\xymatrix@C=0.2cm@R0.2cm{
&&&\circ \ar[dr] &&&&&&\circ \ar[dr]\\
{{\D}^{\bot_1}:} &&\circ \ar[ur]  \ar[dr] &&\cdot  \ar[dr] &&\circ  \ar[dr] &&\cdot  \ar[ur]  \ar[dr] &&\circ \ar[dr]\\
&\circ \ar[ur]  &&\circ \ar[ur]  \ar[dr] &&\cdot \ar[ur] \ar[r] \ar[dr] &\circ \ar[r] &\cdot \ar[ur] \ar[dr] &&\circ \ar[ur] &&\circ\\
&&&&\circ \ar[ur] &&\cdot \ar[ur] &&\circ \ar[ur]
}$$

$$\xymatrix@C=0.2cm@R0.2cm{
&&&\circ \ar[dr] &&&&&&\circ \ar[dr]\\
{\C':} &&\circ \ar[ur]  \ar[dr] &&\cdot  \ar[dr] &&\cdot  \ar[dr] &&\cdot  \ar[ur]  \ar[dr] &&\circ \ar[dr]\\
&\circ \ar[ur]  &&\cdot \ar[ur]  \ar[dr] &&\cdot \ar[ur] \ar[r] \ar[dr] &\circ \ar[r] &\cdot \ar[ur] \ar[dr] &&\circ \ar[ur] &&\cdot\\
&&&&\circ \ar[ur] &&\cdot \ar[ur] &&\circ \ar[ur]
\\} \quad
\xymatrix@C=0.2cm@R0.2cm{
&&&\circ \ar[dr] &&&&&&\circ \ar[dr]\\
{{\C'}^{\bot_1}:} &&\circ \ar[ur]  \ar[dr] &&\cdot  \ar[dr] &&\cdot  \ar[dr] &&\cdot  \ar[ur]  \ar[dr] &&\circ \ar[dr]\\
&\circ \ar[ur]  &&\circ \ar[ur]  \ar[dr] &&\cdot \ar[ur] \ar[r] \ar[dr] &\circ \ar[r] &\cdot \ar[ur] \ar[dr] &&\circ \ar[ur] &&\circ\\
&&&&\circ \ar[ur] &&\cdot \ar[ur] &&\circ \ar[ur]
}$$
$$\xymatrix@C=0.2cm@R0.2cm{
&&&\circ \ar[dr] &&&&&&\circ \ar[dr]\\
{\M:} &&\circ \ar[ur]  \ar[dr] &&\cdot  \ar[dr] &&\circ  \ar[dr] &&\cdot  \ar[ur]  \ar[dr] &&\circ \ar[dr]\\
&\cdot \ar[ur]  &&\cdot \ar[ur]  \ar[dr] &&\cdot \ar[ur] \ar[r] \ar[dr] &\circ \ar[r] &\cdot \ar[ur] \ar[dr] &&\cdot \ar[ur] &&\circ\\
&&&&\circ \ar[ur] &&\cdot \ar[ur] &&\circ \ar[ur]
\\} \quad
\xymatrix@C=0.2cm@R0.2cm{
&&&\circ \ar[dr] &&&&&&\circ \ar[dr]\\
{\mathcal N:} &&\circ \ar[ur]  \ar[dr] &&\cdot  \ar[dr] &&\cdot  \ar[dr] &&\cdot  \ar[ur]  \ar[dr] &&\circ \ar[dr]\\
&\cdot \ar[ur]  &&\cdot \ar[ur]  \ar[dr] &&\cdot \ar[ur] \ar[r] \ar[dr] &\circ \ar[r] &\cdot \ar[ur] \ar[dr] &&\cdot \ar[ur] &&\circ\\
&&&&\circ \ar[ur] &&\cdot \ar[ur] &&\circ \ar[ur]
}$$
$$
\xymatrix@C=0.2cm@R0.2cm{
&&&\circ \ar[dr] &&&&&&\circ \ar[dr]\\
{\M':} &&\circ \ar[ur]  \ar[dr] &&\cdot  \ar[dr] &&\cdot  \ar[dr] &&\cdot  \ar[ur]  \ar[dr] &&\circ \ar[dr]\\
&\cdot \ar[ur]  &&\circ \ar[ur]  \ar[dr] &&\cdot \ar[ur] \ar[r] \ar[dr] &\circ \ar[r] &\cdot \ar[ur] \ar[dr] &&\cdot \ar[ur] &&\circ\\
&&&&\circ \ar[ur] &&\cdot \ar[ur] &&\circ \ar[ur]
}
$$


The heart of $(\C,\C^{\bot_1})$ is the following:

$$
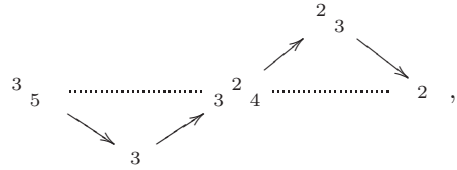
\xymatrix@C=0.4cm@R0.4cm{
&&&{\begin{smallmatrix}
2&\ \\
&3
\end{smallmatrix}}\ar[dr]\\
{\begin{smallmatrix}
3&&\ \\
&5&
\end{smallmatrix}}\ar[dr] \ar@{.}[rr]
&&{\begin{smallmatrix}
&2&\ \\
3&&4
\end{smallmatrix}}\ar[ur] \ar@{.}[rr]
&&{\begin{smallmatrix}
\ &2&\
\end{smallmatrix}},\\
&{\begin{smallmatrix}
\ &3&\
\end{smallmatrix}} \ar[ur]}$$
and the heart of $(\C'^{\bot_1},\M')$ is the following:
$$
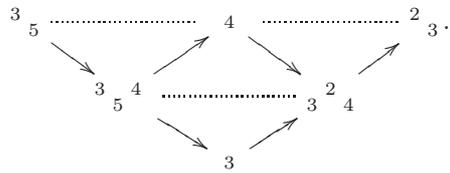
\xymatrix@C=0.4cm@R0.4cm{
{\begin{smallmatrix}
3&\ \\
&5
\end{smallmatrix}} \ar[dr] \ar@{.}[rr] 
&&{\begin{smallmatrix}
\ &4&\ 
\end{smallmatrix}} \ar@{.}[rr] \ar[dr]
&&{\begin{smallmatrix}
2&\ \\
&3
\end{smallmatrix}}.\\
&{\begin{smallmatrix}
3&&4\ \\
&5&
\end{smallmatrix}}\ar[dr] \ar@{.}[rr] \ar[ur]
&&{\begin{smallmatrix}
&2&\ \\
3&&4
\end{smallmatrix}}\ar[ur] \\
&&{\begin{smallmatrix}
\ &3&\
\end{smallmatrix}} \ar[ur]}$$

Obviously they are not equivalent, but if we take localization as in the last section, we get 
$$
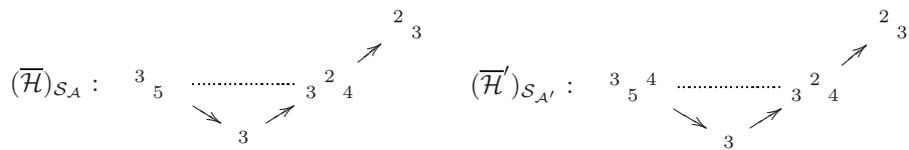
\xymatrix@C=0.2cm@R0.2cm{
&&&&{\begin{smallmatrix}
2&\ \\
&3
\end{smallmatrix}}\\
(\overline \h)_{\s_{\mathcal A}}: &{\begin{smallmatrix}
3&&\ \\
&5&
\end{smallmatrix}}\ar[dr] \ar@{.}[rr]
&&{\begin{smallmatrix}
&2&\ \\
3&&4
\end{smallmatrix}} \ar[ur] \\
&&{\begin{smallmatrix}
\ &3&\
\end{smallmatrix}} \ar[ur]\\}\quad
\xymatrix@C=0.2cm@R0.2cm{
&&&&{\begin{smallmatrix}
2&\ \\
&3
\end{smallmatrix}}
\\
(\overline \h')_{\s_{\mathcal A'}}: &{\begin{smallmatrix}
3&&4\ \\
&5&
\end{smallmatrix}}\ar[dr] \ar@{.}[rr] 
&&{\begin{smallmatrix}
&2&\ \\
3&&4
\end{smallmatrix}} \ar[ur] \\
&&{\begin{smallmatrix}
\ &3&\
\end{smallmatrix}} \ar[ur]}$$

\end{exm}

When we consider the right $\D$-mutation of $\C$, we can not always get a rigid subcategory $\C'=\CoCone(\D,\C)\cap \D^{\bot_1}$. See the following exmaple in $\mod \Lambda$ of Exmaple \ref{ex1}.

\begin{exm}
$$\xymatrix@C=0.2cm@R0.2cm{
&&&\circ \ar[dr] &&&&&&\circ \ar[dr]\\
{\C:} &&\circ \ar[ur]  \ar[dr] &&\cdot  \ar[dr] &&\circ  \ar[dr] &&\cdot  \ar[ur]  \ar[dr] &&\circ \ar[dr]\\
&\circ \ar[ur]  &&\cdot \ar[ur]  \ar[dr] &&\cdot \ar[ur] \ar[r] \ar[dr] &\circ \ar[r] &\cdot \ar[ur] \ar[dr] &&\cdot \ar[ur] &&\circ,\\
&&&&\circ \ar[ur] &&\cdot \ar[ur] &&\circ \ar[ur]
\\} \quad
\xymatrix@C=0.2cm@R0.2cm{
&&&\circ \ar[dr] &&&&&&\circ \ar[dr]\\
{\D:} &&\circ \ar[ur]  \ar[dr] &&\cdot  \ar[dr] &&\cdot  \ar[dr] &&\cdot  \ar[ur]  \ar[dr] &&\circ \ar[dr]\\
&\circ \ar[ur]  &&\cdot \ar[ur]  \ar[dr] &&\cdot \ar[ur] \ar[r] \ar[dr] &\circ \ar[r] &\cdot \ar[ur] \ar[dr] &&\cdot \ar[ur] &&\circ,\\
&&&&\circ \ar[ur] &&\cdot \ar[ur] &&\circ \ar[ur]
}$$
$$\xymatrix@C=0.2cm@R0.2cm{
&&&\circ \ar[dr] &&&&&&\circ \ar[dr]\\
{\C':} &&\circ \ar[ur]  \ar[dr] &&\cdot  \ar[dr] &&\cdot  \ar[dr] &&\cdot  \ar[ur]  \ar[dr] &&\circ \ar[dr]\\
&\circ \ar[ur]  &&\circ \ar[ur]  \ar[dr] &&\cdot \ar[ur] \ar[r] \ar[dr] &\circ \ar[r] &\cdot \ar[ur] \ar[dr] &&\cdot \ar[ur] &&\circ.\\
&&&&\circ \ar[ur] &&\cdot \ar[ur] &&\circ \ar[ur]
}$$
\end{exm}

But in the following cases, we can always get a rigid subcategory $\C'$.

\begin{prop}
Let $\C$, $\D$ be subcategories satisfying condition (RCP) and $\D\subset \C$. Let $\C'=\CoCone(\D,\C)\cap \D^{\bot_1}$. In the folloiwng cases, $\C'$ is rigid:
\begin{itemize}
\item[(a)] $\B$ is a Krull-Schmidt, $k$-linear, Hom-finite, 2-Calabi-Yau triangulated category, $\D$ is funtorially finite;

\item[(b)] $\EE^2(\C,\D)=0$.
\end{itemize}

\end{prop}

\begin{proof}
(a) By the assumptions, we have a cotorsion pair $(\C,\C^{\bot_1})$. By \cite[Theorem 3.3]{ZZ}, we get a cotorsion pair $(\C',\C'^{\bot_1})$ such that $\C'\cap \C'^{\bot_1}=\C'$, which implies $\C'$ is rigid.

(b) Let $C'\in \C'$, it amdits a conflation $(\star) \xymatrix{C' \ar@{ >->}[r] &D \ar@{->>}[r]^{f} &C}$ where $D\in \D$ and $C\in \C$. 
By applying $\EE(\C,-)$ to conflation $(\star)$, we get an exact sequence $0=\EE(\C,C)\to \EE^2(\C,C')\to \EE^2(\C,D)=0$, hence $\EE^2(\C,C')=0$. Since $C'$ is arbitary, we get $\EE^2(\C,\C')=0$. Then we apply $\EE(-,\C')$ to conflation $(\star)$, we get an exact sequence $0=\EE(D,\C')\to \EE(C',\C')\to \EE^2(C,\C')=0$, hence $\EE(C',\C')=0$. Since $C'$ is arbitary, we get $\C'$ is rigid.
\end{proof}





\begin{thebibliography}{BBD}

\bibitem[AN]{AN}
N. Abe, H. Nakaoka.
\newblock General heart construction on a triangulated category (II): Associated cohomological functor.
\newblock Appl. Categ. Structures 20 (2012), no. 2, 162--174.

\bibitem[Au]{Au}
M. Auslander.
\newblock Coherent functors.
\newblock 1966 Proc. Conf. Categorical Algebra (La Jolla, Calif., 1965) pp. 189--231 Springer, New York.


\bibitem[BBD]{BBD}
A. A. Beilinson, J. Bernstein, P. Deligne.
\newblock Faisceaux pervers.
\newblock Analysis and topology on singular spaces, I (Luminy 1981), 5--171, Ast\'{e}risque, 100, Soc. Math. France, Pairs, 1982.


\bibitem[BM]{BM}
A. B. Buan, R. J. Marsh.
\newblock From triangulated categories to module categories via localisation.
\newblock Trans. Amer. Math. Soc. 365 (2013), no. 6, 2845-2861

\bibitem[BM2]{BM2}
A. B. Buan, R. J. Marsh.
\newblock From triangulated categories to module categories via localisation II: calculus of fractions.
\newblock J. Lond. Math. Soc. (2) 87 (2013), no. 2, 643.



\bibitem[DL]{DL}
L. Demonet, Y. Liu.
\newblock Quotients of exact categories by cluster tilting subcategories as module categories.
\newblock Journal of Pure and Applied Algebra. 217 (2013), 2282--2297.


\bibitem[GZ]{GZ}
P. Gabriel, M. Zisman.
\newblock Calculus of fractions and homotopy theory.
\newblock Ergebnisse der Mathematik und ihrer Grenzgebiete, Band 35 Springer-Verlag New York, Inc., New York 1967.


\bibitem[Ha]{Ha}
D. Happel.
\newblock Triangulated categories in the representation theory of finite-dimensional algebras.
\newblock London Mathematical Society Lecture Note Series, 119. Cambridge University Press, Cambridge, 1988. x+208 pp.





\bibitem[IY]{IY}
O. Iyama, Y. Yoshino.
\newblock Mutation in triangulated categories and rigid Cohen-Macaulay modules.
\newblock Invent. Math. 172 (2008), no. 1, 117--168.

\bibitem[KZ]{KZ}
S. Koenig, B. Zhu.
\newblock From triangulated categories to abelian categories: cluster tilting in a general framework.
\newblock Math. Z. 258 (2008), no. 1, 143--160.

\bibitem[L1]{L1}
Y. Liu.
\newblock Hearts of twin cotorsion pairs on exact categories.
\newblock J. Algebra. 394 (2013), 245--284.

\bibitem[L2]{L2}
Y. Liu.
\newblock Half exact functors associated with general hearts on exact categories.
\newblock arXiv: 1305.1433.

\bibitem[LN]{LN}
Y. Liu, H. Nakaoka.
\newblock Hearts of twin Cotorsion pairs on extriangulated categories.
\newblock arXiv: 1702.00244.

\bibitem[MP]{MP}
R. J. Marsh, Y. Palu.
\newblock Nearly Morita equivalences and rigid objects.
\newblock Nagoya Math. J., (2016), 1--36.

\bibitem[N1]{N1}
H. Nakaoka.
\newblock General heart construction on a triangulated category (I): unifying $t$-structures and cluster tilting subcategories.
\newblock Appl. Categ. Structures 19 (2011), no. 6, 879--899.

\bibitem[N2]{N2}
H. Nakaoka.
\newblock General heart construction for twin torsion pairs on triangulated categories.
\newblock Journal of Algebra. 374 (2013), 195--215.

\bibitem[NP]{NP}
H. Nakaoka, Y. Palu.
\newblock Mutation via hovey twin cotorsion pairs and model structures in extriangulated categories.
\newblock arXiv:1605.05607.





\bibitem[ZZ]{ZZ}
Y. Zhou, B. Zhu.
\newblock Mutation of torsion pairs in triangulated categories and its geometric realization.
\newblock arXiv:1105. 3521.

\end{thebibliography}
\end{document}